\documentclass[10pt,draft,twoside]{amsart}
\usepackage{amssymb}
\usepackage{latexsym}
\usepackage{amsfonts}
\usepackage{amsmath}

\DeclareMathOperator{\Aut}{Aut}
\DeclareMathOperator{\rep}{rep}
\DeclareMathOperator{\Rep}{Rep}
\DeclareMathOperator{\Num}{Num}
\DeclareMathOperator{\Inj}{Inj}

\DeclareMathOperator{\supp}{supp}

\DeclareMathOperator{\fix}{fix}

\DeclareMathOperator{\Diag}{Diag}

\DeclareMathOperator{\All}{All}

\DeclareMathOperator{\wt}{wt}

\newtheorem{theorem}{Theorem}[section]

\newtheorem{lemma}[theorem]{Lemma}
\newtheorem{corollary}[theorem]{Corollary}

\theoremstyle{definition}

\newtheorem{remark}[theorem]{Remark}
\newtheorem{definition}[theorem]{Definition}

\renewcommand{\leq}{\leqslant}
\renewcommand{\geq}{\geqslant}

\numberwithin{equation}{section}

\begin{document}

\title[Diagonally Neighbour Transitive Codes and Frequency Permutation
  Arrays] 
      {Diagonally Neighbour Transitive Codes\\and Frequency Permutation
        Arrays} 
\author{Neil I. Gillespie and Cheryl E. Praeger}
\address{[Gillespie and Praeger] Centre for the Mathematics of Symmetry and Computation\\
School of Mathematics and Statistics\\
The University of Western Australia\\
35 Stirling Highway, Crawley\\
Western Australia 6009\\
[Praeger] also affiliated with King Abdulaziz University, Jeddah, Saudi Arabia.}

\email{neil.gillespie@graduate.uwa.edu.au,
  cheryl.praeger@uwa.edu.au}

\begin{abstract}  Constant composition codes have been proposed as
  suitable coding schemes to solve the narrow band and impulse noise 
  problems associated with powerline communication.  In particular, a
  certain class of constant composition codes called frequency
  permutation arrays have been suggested as ideal, in some sense, for
  these purposes.  In this paper we characterise a family of
  neighbour transitive codes in Hamming graphs in which frequency permutation 
  arrays play a central rode.  We also classify all the permutation codes 
  generated by groups in this family.
\end{abstract}

\thanks{{\it Date:} draft typeset \today
 }

\maketitle

\section{Introduction}

Powerline communication has been proposed as a solution to
the ``last mile problem'' in the delivery of fast and reliable
telecommunications at the lowest cost \cite{hanvinck,stateoftheart}.  Any
coding scheme designed for powerline communication must maintain a constant power output, 
while at the same time combat both \emph{permanent narrow band noise} and \emph{impulse noise}, as
well as the usual white Gaussian/background noise \cite{chu1,hanvinck,stateoftheart}.  
Addressing the last of these, the authors introduced
\emph{neighbour transitive codes} (see below) as a group
theoretic analogue to the assumption that white Gaussian noise affects
symbols in codewords independently at random \cite{ngpaper} - an
assumption often made in the theory of error-correcting codes 
\cite[p.5]{pless}.  To deal with the other noise considerations in powerline
communication, \emph{constant composition codes} (CCC) have been proposed
as suitable coding schemes \cite{chu1,chu2} - these codes are of length $m$ over
an alphabet of size $q$ and have the property that each codeword
has $p_i$ occurrences of the $i$th letter of the alphabet, where
the $p_i$ are positive integers such that $\sum p_i=m$.  It is also
suggested in \cite{chu1} that constant composition codes where the
$p_i$ are all roughly $m/q$ are particularly well suited
for powerline communication.  Constant composition codes where each
letter occurs $m/q$ times in each codeword are called \emph{frequency
  permutation arrays}, and were introduced in \cite{sophie1}.  In this paper we
characterise a family of neighbour transitive codes in which frequency permutation arrays play a central role, and
we classify the subfamily consisting of \emph{permutation codes} generated by groups (each of
which is associated with a $2$-transitive permutation group).  

We consider a code of length $m$ over an alphabet $Q$ of
size $q$ to be a subset of the vertex set of the Hamming graph $\Gamma=H(m,q)$,
which has automorphism group $\Aut(\Gamma)\cong S_q^m\rtimes S_m$.  
We define the \emph{automorphism group of a code $C$} to be the setwise
stabiliser of $C$ in $\Aut(\Gamma)$, and we denote it by $\Aut(C)$ (and note that
this is a more general notion than is sometimes used in the literature).  We
define the \emph{the set of neighbours of $C$} to be the set $C_1$ of vertices in
$\Gamma$ that are not codewords, but are adjacent to at least one codeword in $C$.  
We say $C$ is \emph{$X$-neighbour transitive}, or simply \emph{neighbour
transitive}, if there exists a group $X$ of automorphisms such that both $C$ and $C_1$ are $X$-orbits.

Let $\alpha$ be a vertex in $H(m,q)$, and suppose $\{a_1,\ldots,a_k\}$ is
the set of letters that occur in $\alpha$.  The \emph{composition of $\alpha$} 
is the set \begin{equation}\label{compo}Q(\alpha)=\{(a_1,p_1),\ldots,(a_k,p_k)\},\end{equation} 
where the $p_i$ are positive integers and there are exactly $p_i$ occurrences of
the letter $a_i$ in the codeword $\alpha$.  Also let $\mathcal{I}(\alpha)=\{p_1,\ldots,p_k\}$, which can be a multi-set.  
It follows from the definition that, for a constant composition code, $k=q$ and $Q(\alpha)=Q(\beta)$ for all codewords $\alpha,\beta$.
As such, we can talk of the \emph{composition of a constant composition code}, 
which is equal to $Q(\alpha)$ for each codeword $\alpha$.  Now, for a set $\mathcal{I}$ of $k$ positive integers that
sum to $m$, with $k\leq q$, let $\Pi(\mathcal{I})$ be the set of vertices $\alpha$ in $H(m,q)$ with $\mathcal{I}(\alpha)=\mathcal{I}$.    
Then, for any constant composition code $C$,
there exists a set $\mathcal{I}$ of $q$ positive integers such that $C\subseteq\Pi(\mathcal{I})$.  

As automorphisms of a CCC must leave its composition invariant, it is 
natural to ask what types of automorphisms might do this, particularly as we are interested
in neighbour transitive CCC's.  The group $S_q$ (which we identify with the Symmetric group of $Q$) induces a 
faithful action on the vertices of $\Gamma$ in which elements of $S_q$ act naturally on each of the $m$ entries of
a vertex.  We denote the image of $S_q$ under this action by $\Diag_m(S_q)$ (since it is a diagonal subgroup
of the base group $S_q^m$ of $\Aut(\Gamma)$, see \eqref{diagdef}).  
It follows (from Lemma \ref{num}) that $\Pi(\mathcal{I})$ is left invariant under $\Diag_m(S_q)$. 
Similarly, the group $L$ of all permutations of entries fixes $\Pi(\mathcal{I})$ setwise.
(This holds because any permutation of the entries of a vertex $\alpha$ is a rearrangement
of the letters occurring in $\alpha$, leaving the composition $Q(\alpha)$ unchanged.)
Moreover, the group $\langle \Diag_m(S_q),L\rangle=\Diag_m(S_q)\rtimes L$ is the largest subgroup
of $\Aut(\Gamma)$ that leaves invariant $\Pi(\mathcal{I})$ for all $\mathcal{I}$ (for example
no other element of $\Aut(\Gamma)$ fixes $\Pi(\{m\})$).
Hence it is natural to ask which CCC's are fixed setwise by the group $\Diag_m(S_q)\rtimes L$,
or more specifically, which are $X$-neighbour transitive with $X\leq\Diag_m(S_q)\rtimes L$.  This leads to
the following definition.  
\begin{definition} A code $C$ in $H(m,q)$ is \emph{diagonally $X$-neighbour transitive}, or simply
\emph{diagonally neighbour transitive}, if it is $X$-neighbour transitive for some $X\leq\Diag_m(S_q)\rtimes L$.
\end{definition}
\noindent Our first major result characterises diagonally neighbour transitive codes,
and shows that diagonally neighbour transitive CCC's are necessarily frequency
permutation arrays.               

\begin{theorem}\label{main1} Let $C$ be a diagonally neighbour
  transitive code in $H(m,q)$.  Then either $C$ is a frequency permutation
  array; $C=\{(a,\ldots,a)\}$ for some letter $a$; or $C$ is one of the codes
  described in Definition \ref{codedef} (i), (ii) or (iii), none of
  which is a constant composition code.    
\end{theorem}

Theorem \ref{main1} gives us a nice characterisation of diagonally neighbour
transitive codes, but it does not provide us with any examples of neighbour transitive frequency
permutation arrays.  We consider \emph{permutation codes} to find examples of such codes. 
By identifying the alphabet $Q$ with the set $\{1,\ldots,q\}$, any permutation $t\in S_q$ can be
associated with the $q$-tuple $\alpha(t)$ in $H(q,q)$, which has $i$th entry equal to the image of $i$ under
$t$.  For example, if $q=3$ and $t=(1,2,3)$, then $\alpha(t)=(2,3,1)$.  
For a subset $T$ of $S_q$, we define $C(T)=\{\alpha(t)\,:\,t\in T\}$, called the \emph{permutation code
generated by $T$}, and $N_{S_q}(T)=\{x\in S_q\,:\,T^x=T\}$.  

\begin{theorem}\label{permiff} Let $T$ be a subgroup of $S_q$.  Then the permutation code
$C(T)$ is diagonally neighbour transitive in $H(q,q)$ if and only if $N_{S_q}(T)$ is $2$-transitive.
Moreover, for any positive integer $p$ and diagonally neighbour transitive code
$C(T)$, the code $\Rep_p(C(T))$, given in (\ref{genrepc}), is a diagonally neighbour transitive
frequency permutation array in $H(pq,q)$.
\end{theorem}

In Section \ref{secdef} we introduce the required definitions and some
preliminary results.  Then, in Section \ref{ntrex}, we give some examples
of diagonally neighbour transitive codes in $H(m,q)$.  Finally, we prove
Theorems \ref{main1} and \ref{permiff} in Sections \ref{codewgrps} and
\ref{sec1codect} respectively.      

\section{Definitions and Preliminaries}\label{secdef}

Any code of length $m$ over an alphabet $Q$ of size $q$ can be
embedded in the vertex set of the \emph{Hamming graph}.  The Hamming graph $\Gamma=H(m,q)$ 
has vertex set $V(\Gamma)$, the set of $m$-tuples with entries from $Q$, and an edge exists between two
vertices if and only if they differ in precisely one entry.
Throughout we assume that $m,q\geq 2$.  The automorphism group of $\Gamma$, which we
denote by $\Aut(\Gamma)$, is the semi-direct product
$B\rtimes L$ where $B\cong S_q^m$ and $L\cong S_m$, see \cite[Theorem 9.2.1]{distreg}.    Let
$g=(g_1,\ldots, g_m)\in B$, $\sigma\in L$ and
$\alpha=(\alpha_1,\ldots,\alpha_m)\in V(\Gamma)$. Then $g$ and $\sigma$
act on $\alpha$ in the following way: \[\begin{array}{ccc}
\alpha^g=(\alpha_1^{g_1},\ldots,\alpha_m^{g_m}),&\,\,\,&\alpha^\sigma=(\alpha_{1\sigma^{-1}},\ldots,\alpha_{m\sigma^{-1}}). \end{array}\]
For any subgroup $T$ of $S_q$, we define the following subgroup of
$B$: \begin{equation}\label{diagdef}\Diag_m(T)=\{(h,\ldots,h)\in B\,:\,h\in T\}.\end{equation}  Let 
$M=\{1,\ldots,m\}$, and view $M$ as the set of vertex entries of
$H(m,q)$.  Let $0$ denote a distinguished element of the alphabet $Q$.
For $\alpha\in V(\Gamma)$, the \emph{support of $\alpha$} is the set
$\supp(\alpha)=\{i\in M\,:\,\alpha_i\neq 0\}$.  The \emph{weight of
  $\alpha$} is defined as $\wt(\alpha)=|\supp(\alpha)|$.  For all pairs of
vertices $\alpha,\beta\in V(\Gamma)$, the \emph{Hamming distance} between
$\alpha$ and $\beta$, denoted by $d(\alpha,\beta)$, is defined to be
the number of entries in which the two vertices differ.  We let
$\Gamma_k(\alpha)$ denote the set of vertices in $H(m,q)$ that are at
distance $k$ from $\alpha$.  For $a_1,\ldots,a_k\in Q$ and positive integers
$p_1,\ldots,p_k$ such that $\sum p_i=m$, we let
$(a_1^{p_1},a_2^{p_2},\ldots,a_k^{p_k})$ denote the
vertex
$$(\underbrace{a_1,\ldots,a_1}_{p_1},\underbrace{a_2,\ldots,a_2}_{p_2},\ldots,
\underbrace{a_k,\ldots,a_k}_{p_k})\in 
V(\Gamma)$$ Let $\alpha=(\alpha_1,\ldots,\alpha_m)\in V(\Gamma)$.  For $a\in Q$
we let $\nu(\alpha,i,a)\in V(\Gamma)$ denote the vertex with $j$th entry 
\[\nu(\alpha,i,a)|_j=\left\{\begin{array}{ll}      
 \alpha_j&\textnormal{if $j\neq i$}\\ a &\textnormal{if
   $j=i$.} \end{array}\right.\]  
We note that if $\alpha_i=a$ then $\nu(\alpha,i,a)=\alpha$, otherwise
$\nu(\alpha,i,a)\in\Gamma_1(\alpha)$.  Throughout this paper whenever
we refer to $\nu(\alpha,i,a)$ as a \emph{neighbour of $\alpha$}, or
being adjacent to $\alpha$, we mean that $a\in
Q\backslash\{\alpha_i\}$.  The following straight forward result describes the
action of automorphisms of $\Gamma$ on vertices of this form.

\begin{lemma}\label{neigact} Let $\alpha=(\alpha_1,\ldots,\alpha_m)\in
  V(\Gamma)$, $a\in Q$, and $x=(h_1,\ldots,h_m)\sigma\in\Aut(\Gamma)$.
  Then $\nu(\alpha,i,a)^x=\nu(\alpha^x,i^\sigma,a^{h_i})$, and is adjacent to 
  $\alpha^x$ if and only if $\nu(\alpha,i,a)$ is adjacent to $\alpha$.  
\end{lemma}

For a code $C$ in $H(m,q)$, the \emph{minimum distance, $\delta$,
  of C} is the smallest distance between distinct codewords of $C$.
For any $\gamma\in V(\Gamma)$, we define 
$$d(\gamma,C)=\min\{d(\gamma,\beta)\,:\,\beta\in C\}.$$ to be the 
\emph{distance of $\gamma$ from $C$}.  The \emph{covering radius of $C$}, 
which we denote by $\rho$, is the maximum distance that any vertex in $H(m,q)$ is from $C$.  
We let $C_i$ denote the set of vertices that are distance $i$ from $C$, and deduce,
for $i\leq \lfloor (\delta-1)/2\rfloor$, that $C_i$ is the disjoint 
union of $\Gamma_i(\alpha)$ as $\alpha$ varies over $C$.  
Furthermore, $C_0=C$ and $\{C,C_1,\ldots,C_\rho\}$ forms a
partition of $V(\Gamma)$ called the \emph{distance partition of $C$}. 
In particular, the \emph{complete code} $C=V(\Gamma)$ has covering
radius $0$ and trivial distance partition $\{C\}$; and if $C$ is not
the complete code, we call the non-empty subset $C_1$ the \emph{set of
  neighbours of $C$}.  Let $C$ and $C'$ be codes in $H(m,q)$.  We say
$C$ and $C'$ are \emph{equivalent} if there exists $x\in\Aut(\Gamma)$
such that $C^x=C'$, and if $C'=C$ we call $x$ an automorphism of $C$.  
Recall, the automorphism group of $C$, denoted by $\Aut(C)$, 
is the setwise stabiliser of $C$ in $\Aut(\Gamma)$.  

Let $C$ be a code in $H(m,q)$ with distance partition $\{C,C_1,\ldots,C_\rho\}$.  
As we defined in the introduction, we say $C$ is $X$-neighbour transitive if there exists $X\leq\Aut(\Gamma)$ such that 
$C_i$ is an $X$-orbit for $i=0,1$.  If there exists $X\leq\Aut(\Gamma)$ such
that $C_i$ is an $X$-orbit for $i=0,\ldots,\rho$, we say $C$ is \emph{$X$-completely
transitive}, or simply \emph{completely transitive.}

\begin{remark}  The reader should note that the definition of
  neighbour transitivity given in \cite{ngpaper} is more general than
  the one given here in that it only requires $C_1$ to be an $X$-orbit.
  However, it is not unreasonable to use the definition given here because if
  $\delta\geq 3$ and $C_1$ is an $X$-orbit with $X\leq\Aut(C)$, then $X$
  necessarily acts transitively on $C$, and furthermore, it is shown in
  \cite{ngpaper} that an automorphism group that fixes $C_1$ setwise often has to also fix $C$
  setwise. Note also that completely transitive codes are necessarily
  neighbour transitive.
\end{remark}

\begin{lemma}\label{distpart} Let $C$ be a code with distance partition
  $\mathcal{C}=\{C,C_1,\ldots,C_\rho\}$ and $y\in\Aut(\Gamma)$.  Then
  $C_i^y:=(C_i)^y=(C^y)_i$ for each $i$.  In particular, the code
  $C^y$ has distance partition $\{C^y,C_1^y,\ldots,C^y_\rho\}$, and
  $\mathcal{C}$ is $\Aut(C)$-invariant.  Moreover, $C$ is
  $X$-neighbour (completely) transitive if and only if $C^y$ is
  $X^y$-neighbour (completely) transitive.      
\end{lemma}

\begin{proof}  Let $\beta\in C_i$.  Then there exists $\alpha\in C$
  such that $d(\beta,\alpha)=i$.  Since automorphisms preserve
  adjacency it follows that $d(\beta^y,\alpha^y)=i$.  Thus
  $d(\beta^y,C^y)\leq i$.  The same argument shows that if
  $j=d(\beta^y,C^y)$ then
  $i=d(\beta,C)=d((\beta^y)^{y^{-1}},(C^{y})^{y^{-1}})\leq j$, and
  hence $d(\beta^y,C^y)=i$.  Thus $(C_i)^y\subseteq (C^y)_i$.  A
  similar argument shows that $(C^y)_i\subseteq (C_i)^y$.  Hence
  $(C_i)^y=(C^y)_i$. Therefore, without ambiguity, we can denote this
  set by $C_i^y$.  Thus the distance partition of $C^y$ is
  $\{C^y,C_1^y\ldots,C_\rho^y\}$.  In particular, if $y\in\Aut(C)$, it
  follows that $(C_i)^y=(C^y)_i=C_i$ for each $i$.  That is
  $\mathcal{C}$ is $\Aut(C)$-invariant.  Finally, $C$ is
  $X$-neighbour (completely) transitive if and only if $C_i$ is an
  $X$-orbit for $i=0,1$ ($i=0,\ldots,\rho$), which holds if and only
  if $C_i^y$ is an $X^y$-orbit for $i=0,1$ ($i=0,\ldots,\rho$).
\end{proof}

Let $C$ be a code with covering
  radius $\rho$ and let $s\in\{0,\ldots,\rho\}$. 
  As in \cite[p. 346]{distreg}, we say $C$ is \emph{$s$-regular} if
  for each vertex $\gamma\in C_i$, with $i=0,\ldots,s$, and integer
  $k=0,\ldots,m$, the number of codewords at distance $k$ from
  $\gamma$ depends only on $i$ and $k$, and is independent of the
  choice of $\gamma\in C_i$.  If $s=\rho$ we say $C$ is
  \emph{completely regular}.      

\begin{remark}\label{regrem} It is known that completely transitive
  codes are necessarily completely regular \cite[Lemma 2.1]{giupra}.
  Similarly, because automorphisms preserve adjacency, it is straight
  forward to show that any neighbour transitive code is necessarily $1$-regular.
     
\end{remark} 

\begin{lemma}\label{partfact}  Let $C$ be a completely regular code in
  $H(m,q)$ with distance partition
  $\{C,C_1,\ldots,C_\rho\}$.  Then $C_\rho$ is completely
  regular with distance partition
  $\{C_\rho,C_{\rho-1},\ldots,C_1,C\}$; and $\Aut(C)=\Aut(C_\rho)$.
  Furthermore, $C$ is $X$-completely transitive if and only if
  $C_\rho$ is $X$-completely transitive.       
\end{lemma}

\begin{proof}  The fact that $C_\rho$ is completely regular with
  distance partition $\{C_\rho,C_{\rho-1},\ldots,C\}$ is given in
  \cite{neum}.  It then follows from Lemma \ref{distpart} that 
  $\Aut(C)=\Aut(C_\rho)$.  By definition, $C$ is $X$-completely
  transitive if and only if each $C_i$ is an $X$-orbit, which therefore holds
  if and only if $C_\rho$ is $X$-completely transitive.     
\end{proof}

For $\alpha\in V(\Gamma)$, recall $Q(\alpha)$, the
composition of $\alpha$ defined in (\ref{compo}).  For each distinct
$p_i$ that appears in $Q(\alpha)$ we want to register the number of
distinct letters that appear $p_i$ times.  We
let $$\Num(\alpha)=\{(p_1,s_1),\ldots,(p_j,s_j)\}$$ where 
$(p_i,s_i)$ means that $s_i$ distinct letters appear $p_i$ times in
$\alpha$.  We note that $\sum s_i=k$, the number of distinct letters
that occur in $\alpha$.      

\begin{lemma}\label{num} Let $\alpha\in V(\Gamma)$ with
  $Q(\alpha)=\{(a_1,p_1),\ldots,(a_k,p_k)\}$ and $x=(h,\ldots,h)\sigma\in
\Diag_m(S_q)\rtimes L$.  Then 
  $Q(\alpha^x)=\{(a_1^h,p_1),\ldots,(a_k^h,p_k)\}$ and
$\Num(\alpha^x)=\Num(\alpha)$.
\end{lemma}

\begin{proof}  Let $\alpha=(\alpha_1,\ldots,\alpha_m)$ and $a\in Q$.  Note that
  $\alpha_i=a$ if and only if $\alpha_i^h=a^h$, and that
  $\alpha_i^h=\alpha^x|_{i^\sigma}$.  Therefore for every occurrence
  of $a$ in $\alpha$ there is a corresponding occurrence of $a^h$ in
  $\alpha^x$.  Thus $Q(\alpha^x)=\{(a_1^h,p_1),\ldots,(a_k^h,p_k)\}$.
  We note that $\{p_1,\ldots,p_k\}$ is left invariant by the action of
  $x$ on $\alpha$.  Therefore $\Num(\alpha)=\Num(\alpha^x)$.
\end{proof}

\begin{corollary}\label{specprop} Let $C$ be a diagonally $X$-neighbour
transitive
  code, and let $\nu\in C_i$ for $i=0,1$.  Then $\Num(\nu')=\Num(\nu)$ for all
  $\nu'\in C_i$.  If in addition $X\leq L$, then $Q(\nu')=Q(\nu)$
  for all $\nu'\in C_i$.       
\end{corollary}

For a positive integer $p$, we can identify the vertex set of the 
Hamming graph $\Gamma^{(p)}=H(mp,q)$ with the set of arbitrary $p$-tuples of
vertices from $\Gamma=H(m,q)$.  For a group $X\leq\Aut(\Gamma)$, we let $(x,\sigma)\in X\times S_{p}$ act on the
vertices of $\Gamma^{(p)}$ in the following way:
$$(\alpha_1,\ldots,\alpha_p)^{(x,\sigma)}=(\alpha^x_{1\sigma^{-1}},\ldots,
\alpha^x_{p\sigma^{-1}}),$$ where $\alpha_1,\ldots,\alpha_p\in V(\Gamma)$.
For $\alpha\in V(\Gamma)$, we let $\rep_{p}(\alpha)=(\alpha,\ldots,\alpha)\in
V(\Gamma^{(p)})$, and for a code $C$ in $\Gamma$ with minimum distance $\delta$ we let 
\begin{equation}\label{genrepc}\Rep_{p}(C)=\{\rep_p(\alpha)\,:\,\alpha\in
C\},\end{equation} 
which is a code in $\Gamma^{(p)}$ with minimum distance $p\delta$.  
It follows that $\rep_p(\alpha)^{(x,\sigma)}=\rep_p(\alpha^x)$, and so $C$ is 
an $X$-orbit if and only if $\Rep_p(C)$ is an $(X\times S_p)$-orbit.  
For $\alpha,\nu\in V(\Gamma)$ we let $\mu(\rep_p(\alpha),i,\nu)$ denote the
vertex constructed by 
changing the $i$th vertex entry of $\rep_p(\alpha)$ from $\alpha$ to $\nu$.  It
follows that $\nu\in\Gamma_1(\alpha)$ if 
and only if $\mu(\rep_{p}(\alpha),i,\nu)\in\Gamma_1(\rep_p(\alpha))$, and that
$\mu(\rep_p(\alpha),i,\nu)^{(x,\sigma)}=\mu(\rep_p(\alpha^x),i^\sigma,\nu^x)$.

\begin{lemma}\label{replemma} Let $C$ be an $X$-neighbour transitive code in $\Gamma=H(m,q)$ with
$\delta\geq 2$ such that a stabiliser $X_\alpha$ acts transitively on $\Gamma_1(\alpha)$ for some $\alpha\in C$.  Then
$\Rep_p(C)$ is $(X\times S_p)$-neighbour transitive in $H(mp,q)$.  
\end{lemma}

\begin{proof} It follows from the comments above and Lemma \ref{distpart} that
we only need to 
prove the transitivity on the neighbours of $\Rep_p(C)$.  Let
$\nu_1,\nu_2\in\Rep_p(C)_1$.
Then there exist $i,j$ and $\beta,\gamma\in C$ such that
$\nu_1=\mu(\rep_p(\beta),i,\nu_\beta)$ and 
$\nu_2=\mu(\rep_p(\gamma),j,\nu_\gamma)$ for some adjacent vertices
$\nu_\beta,\nu_\gamma$ of $\beta,\gamma$ in $\Gamma$ 
respectively.  There exists $x\in X$ such that $\beta^x=\gamma$, so
$\nu_1^{(x,1)}=\mu(\rep_p(\gamma),i,\nu_\beta^x)$, and 
$\nu_\beta^x\in\Gamma_1(\gamma)$ since adjacency is preserved by $x$ in
$\Gamma$.  As $X$ acts 
transitively on $C$, and because $X_\alpha$ acts transitively on
$\Gamma_1(\alpha)$, there exists $y\in X_\gamma$ such that 
$\nu_\beta^{xy}=\nu_\gamma$.  By choosing $\sigma\in S_p$ such that
$i^\sigma=j$, 
we deduce that $\nu_1^{(xy,\sigma)}=\nu_2$.    
\end{proof}

Let $C$ be a neighbour transitive code in $H(m,q)$ with $\delta=1$.  Let
$\alpha,\beta\in C$ such that $d(\alpha,\beta)=1$, and 
$\nu\in\Gamma_1(\alpha)\cap C_1$ (such a vertex exists by the transitivity on
$C$).  It follows that
$\nu_1=\mu(\rep_p(\alpha),1,\nu)$,
$\nu_2=\mu(\rep_p(\alpha),1,\beta)\in\Rep_p(C)_1$ in $H(pq,q)$.  However, 
there does not exist $x\in\Aut(C)$ such that $\beta^x=\nu$ because $\Aut(C)$
fixes $C$ setwise, and so
$\nu_1$ and $\nu_2$ are not contained in the same $(\Aut(C)\times S_p)$-orbit. 
Thus the condition
that $\delta\geq 2$ in Lemma \ref{replemma} is essential.  

\section{Examples of Neighbour Transitive Codes}\label{ntrex}

In this section we define four infinite families of codes and prove
that all codes in these families are neighbour transitive.  In Section
\ref{codewgrps}, 
we use these codes to classify diagonally neighbour transitive codes in
$\Gamma=H(m,q)$.  
In all cases $m>1$.  

\begin{definition}\label{codedef} (i) The \emph{repetition code in $H(m,q)$} is 
$$\Rep(m,q)=\{(a^m)\,:\,a\in Q\}=\{\alpha\in V(\Gamma)\,:\,\Num(\alpha)=\{(m,1)\}\}.$$   
  
  \noindent (ii) Let $m<q$, and
  define \begin{align*}\Inj(m,q)=&\,\,\{(\alpha_1,\ldots,\alpha_m)\in
  V(\Gamma)\,:\,\alpha_i\neq\alpha_j\textnormal{ for }i\neq
  j\}\\=&\,\,\{\alpha\in V(\Gamma)\,:\,\Num(\alpha)=\{(1,m)\}\,\}.\end{align*} 
 
  \noindent (iii) Let $m$ be odd with $m\geq 3$ and $q=2$, and define,
  in $\Gamma=H(m,2)$, \begin{align*}W([m/2],2)=&\,\,\{\alpha\in
  V(\Gamma)\,:\, \wt(\alpha)=(m\pm 1)/2\,\}\\=&\,\,\{\alpha\in
  V(\Gamma)\,:\,\Num(\alpha)=\{((m+1)/2,1),((m-1)/2,1)\}\,\}.\end{align*}
  \noindent (iv) Let $p$ be any positive integer, and let $m=pq$, and
  define $$\All(pq,q)=\{\alpha\in V(\Gamma)\,:\,\Num(\alpha)=\{(p,q)\}\,\}$$ 
\end{definition}

\begin{remark}  The codes $\Inj(m,q)$ are examples of \emph{injection codes}, 
which were recently introduced by Dukes \cite{dukes}.  Note also that $\All(pq,q)$ is the 
largest possible frequency permutation array of length $pq$ over an alphabet of size $q$.  
\end{remark}

\begin{theorem}\label{codethm} Let $C$ be one of the codes in Definition
  \ref{codedef}.  Then $C$ is neighbour transitive with 
  $\Aut(C)=\Diag_m(S_q)\rtimes L$.  Moreover, $C$ has minimum
  distance $\delta=m$, $1$, $1$ and $2$ respectively in (i), (ii),
  (iii), (iv) of Definition \ref{codedef}. 
\end{theorem}
 
\begin{proof} It follows from Lemma \ref{num} that, in all cases,
  $\Aut(C)$ contains $H=\Diag_m(S_q)\rtimes L$, and it is clear
  that the minimum distance of $C$ is as stated.  Moreover, it is easy
  to check that the group $H$ acts transitively on $C$ (again in all
  four cases).  Now, the set $C_1$ of neighbours is
\[C_1=\left\{\begin{array}{ll}      
 \{\nu\in V(\Gamma)\,:\,\Num(\nu)=\{(m-1,1),(1,1)\}\,\}&\textnormal{in case
(i)}\\
\{\nu\in V(\Gamma)\,:\,\Num(\nu)=\{(2,1),(1,m-2)\}\,\}&\textnormal{in case
(ii)}\\ 
\{\nu\in V(\Gamma)\,:\,\Num(\nu)=\{((m+3)/2,1),((m-3)/2,1)\}\,\}&\textnormal{in
case (iii)}\\
\{\alpha\in
V(\Gamma)\,:\,\Num(\alpha)=\{(p+1,1),(p,q-2),(p-1,1)\}\,\}&\textnormal{in
  case (iv)} \end{array}\right.\]   
  (noting that in case (iv) we may have $q=2$), and again in all cases
  it is straight forward to check that $H$ is transitive on $C_1$.
  Thus $C$ is $H$-neighbour transitive.  It remains to prove that
  $\Aut(C)=H$.  Suppose to the contrary that $\Aut(C)$ contains
  $y=(h_1,\ldots,h_m)\sigma$ such that $h_i\neq h_j$ for some $i\neq
  j$.  Since $L\leq H\leq\Aut(C)$, we may assume that $\sigma=1$ 
  and that $h_1\neq h_2$.  Moreover, since $\Diag_m(S_q)\leq\Aut(C)$,
  we may further assume that $h_2=1$, so $h_1\neq 1$.  Let $a,b\in Q$  
  such that $a^{h_1}=b\neq a$.  We consider the cases above separately, and
  in the first two cases arrive at a contradiction by exhibiting a
  codeword $\alpha\in C$ such that $\alpha^y\notin C$.   

  (i) If $C=\Rep(m,q)$ then $(a^m)^y|_1=b$ and $(a^m)^y|_2=a$,
  so $(a^m)^y\notin C$.  

  (ii) If $C=\Inj(m,q)$, then $C$ contains a codeword $\alpha$ with
  $\alpha_1=a$ and $\alpha_2=b$.  However, $\alpha^y$ has
  $\alpha^y|_1=\alpha^y|_2=b$, so $\alpha^y\notin C$.   

  (iii) Let $q=2$, $C=W([m/2],2)$ with $m\geq 3$ and $m$ odd, and
  consider $$C'=\Rep(m,2)=\{{\bf{0}}=(0,\ldots,0),{\bf{1}}=(1,\ldots,1)\}.$$ 
Let
  $\alpha\in V(\Gamma)$ such that $\wt(\alpha)=k$ for $1\leq k\leq 
  m-1$.  Then $d(\alpha,{\bf{0}})=k$ and $d(\alpha,{\bf{1}})=m-k$.  If
  $k\leq (m-1)/2$, then $k\leq 
  m-1-k<m-k$, and so $d(\alpha,C')=k$.  If $k\geq (m+1)/2$, then $k\geq
  m+1-k>m-k$, and so $d(\alpha,C')=m-k$.  It follows that $d(\alpha,C')$
  is maximised when $k=(m-1)/2$ or $k=(m+1)/2$, and in both cases
  $d(\alpha,C')=(m-1)/2$.  Thus $C'$ has covering radius
  $\rho=(m-1)/2$.  It also follows that $$C'_\rho=W([m/2],2)=C.$$  It
  is known that $C'$ is completely transitive and hence completely
  regular \cite[Sec. 2]{famctr}.  Moreover, we have just proved
  that $\Aut(C')=H$.  Therefore, by Lemma \ref{partfact},
  $\Aut(C)=\Aut(C')=H$.    

  (iv) Let $\nu\in V(\Gamma)$ and suppose
  $Q(\nu)=\{(a_1,p_1),\ldots,(a_k,p_k)\}$ with $p_1\geq
  p_2\geq\ldots\geq p_k$.  Then $k\leq q$ and $p_1+\ldots+p_k=m=pq$, and 
  in particular $p_1\geq p$.  There exists $\sigma\in L\leq\Aut(C)$ such that
  $\nu^\sigma=(a_1^{p_1},a_2^{p_2},\ldots,a_k^{p_k})$.  Consider the 
  codeword $\alpha=(a_1^p,a_2^p,\ldots,a_q^p)\in C$.  Then
  $\nu^\sigma$ and $\alpha$ agree in at least the first $p$ entries.
  Therefore $d(\nu^\sigma,\alpha)\leq p(q-1)$ and so
  $d(\nu,C)=d(\nu^\sigma,C)\leq p(q-1)$.  Therefore $\rho\leq p(q-1)$.
  Now consider $\nu=(a,\ldots,a)$ for some $a\in Q$.  It follows from
  the definition of $C$ that $d(\nu,\alpha)=p(q-1)$ for all $\alpha\in
  C$.  Therefore $d(\nu,C)=p(q-1)$ and so $\rho=p(q-1)$.  Moreover,
  $\Rep(m,q)\subseteq C_\rho$.  Now suppose $\nu\in C_{\rho}$ and
  $Q(\nu)=\{(a_1,p_1),\ldots,(a_k,p_k)\}$ with $k\geq 2$ and $p_1\geq
  p$.  There exists $\sigma\in L\leq\Aut(C)$ such that
  $\nu^\sigma=(a_1^p,a_2^{p_2},a_1^{p_1-p},a_3^{p_3},\ldots,a_k^{p_k})$.  
  Since $\sigma\in\Aut(C)$, Lemma \ref{distpart} implies that
  $\nu^\sigma\in C_\rho$ also.  Consider the codeword
  $\alpha=(a_1^p,a_2^p,\ldots,a_q^p)$.  Then $\nu^\sigma$ and $\alpha$
  agree in the first $p+p_2>p$, therefore $d(\nu^\sigma,\alpha)\leq
  pq-(p+1)<p(q-1)$, which is a contradiction as $\nu^\sigma\in
  C_\rho$.  It follows that $C_\rho=\Rep(m,q)$.  In particular, by
  Lemma \ref{distpart}, $\Aut(C)$ leaves $\Rep(m,q)$ invariant and so
  $\Aut(C)$ is contained in $\Aut(\Rep(m,q))$, which we have just
  proved is equal to $H$.
\end{proof}

\noindent The proof of Theorem \ref{codethm} yields the following
immediate corollary.  

\begin{corollary} (i) If $q=2$ and $m\geq 3$ is odd, then $C=W([m/2],2)$ 
 has covering radius $\rho=(m-1)/2$ and $C_\rho=\Rep(m,2)$.  Furthermore, 
 $C$ and $C_\rho$ are completely transitive.

  (ii) If $m=pq$ for some $p$, then $C=\All(pq,q)$ has covering radius
  $\rho=p(q-1)$ and $C_\rho=\Rep(m,q)$.
\end{corollary}

\section{Characterising Diagonally Neighbour Transitive Codes.}\label{codewgrps}

In this section we characterise diagonally neighbour transitive codes in
$\Gamma=H(m,q)$.  
However, before we consider such codes, we first prove some interesting results
about connected
subsets $\Delta$ of $V(\Gamma)$ (that is to say, the
subgraph of $\Gamma$ induced on $\Delta$ is connected). 

\begin{lemma}\label{connected} Let $\Delta$ be a connected subset of
  $V(\Gamma)$.  Let $C$ be a code that is a proper subset of
  $\Delta$.  Then $C_1\cap\Delta\neq\emptyset$. 
\end{lemma}
\begin{proof}  Let $\alpha\in C$ and $\beta\in\Delta\backslash C$.
  Since $\Delta$ is a connected subset, there exists a
  path $$\alpha=\alpha^0,\alpha^1,\ldots,\alpha^\ell=\beta$$ such that
  each $\alpha^i\in\Delta$.  Because $\alpha\in C$ and $\beta\notin
  C$, there is a least $i<\ell$ such that $\alpha^i\in C$ and
  $\alpha^{i+1}\notin C$.  Since $d(\alpha^i,\alpha^{i+1})=1$, it
  follows that $\alpha^{i+1}\in C_1$.
\end{proof}

\begin{lemma}\label{delta1con}  The codes $\Inj(m,q)$ (with $1<m<q$)
  and $W([m/2],2)$ (with $m$ odd and $m\geq 3$) are connected subsets
  of $V(\Gamma)$.
\end{lemma}

\begin{proof}  Firstly we consider $\Delta_1=\Inj(m,q)$.  Let
  $\alpha,\beta\in\Delta_1$.  We shall prove that $\alpha,\beta$ are
  connected by a path in $\Delta_1$ using induction on the distance
  $d(\alpha,\beta)$ in $\Gamma$.  This is true if $d(\alpha,\beta)=1$,
  so assume that $d(\alpha,\beta)=w>1$, and the property holds for
  distances less than $w$.  Let $S=\{k\,:\,\alpha_k=\beta_k\}$, $i\in
  M\backslash S$ and $\alpha^*=\nu(\alpha,i,\beta_i)$.  Then
  $\alpha^*$ is adjacent to $\alpha$ in $\Gamma$.  If
  $\beta_i\neq\alpha_k$ for all $k\in M\backslash(S\cup\{i\})$, then
  $\alpha^*\in\Delta_1$ and $d(\alpha^*,\beta)=w-1$.  Therefore, by
  the inductive hypothesis, $\alpha^*$ and $\beta$ are connected by a
  path in $\Delta_1$ and hence so are $\alpha$ and $\beta$.  Thus we
  may assume that $\beta_i=\alpha_j$ for some $j\in
  M\backslash(S\cup\{i\})$.  We  
  note that $j$ is unique since $\alpha\in\Delta_1$.  Also
  $\alpha_j^*=\alpha_i^*$ and so $\alpha^*\notin\Delta_1$.  Since
  $m<q$, there exists
  $a\in Q\backslash\{\alpha_1,\ldots,\alpha_m\}$.  Let
  $\alpha^{\diamondsuit}=\nu(\alpha,j,a)$.  Then
  $\alpha^{\diamondsuit}\in\Delta_1\cap\Gamma_1(\alpha)$.  If  
  $a=\beta_j$ then 
  $d(\alpha^{\diamondsuit},\beta)=w-1$.  Therefore, by the inductive
  hypothesis, $\alpha^{\diamondsuit}$ and $\beta$ are connected by a
  path in $\Delta_1$ and hence so are $\alpha$ and $\beta$.  If
  $a\neq\beta_j$ then $d(\alpha^{\diamondsuit},\beta)=w$.  In this
  case let $\alpha^{\heartsuit}=\nu(\alpha^{\diamondsuit},i,\beta_i)$.  It
  follows that
$\alpha^{\heartsuit}\in\Delta_1\cap\Gamma_1(\alpha^{\diamondsuit})$
  and $d(\alpha^{\heartsuit},\beta)=w-1$.  Therefore by the inductive
  hypothesis, $\alpha^{\heartsuit}$ and $\beta$ are connected by a
  path in $\Delta_1$ and hence so are $\alpha$ and $\beta$.  Thus
  $\Delta_1$ is connected by induction.

  We now consider the set $\Delta_2=W([m/2],2)$. 
  Let $\alpha,\beta\in\Delta_2$ such that $\wt(\alpha)=\wt(\beta)=(m+1)/2$. 
Furthermore let 
  $\mathcal{S}=\supp(\alpha)\cap\supp(\beta)$,
$\mathcal{J}=\supp(\alpha)\backslash\mathcal{S}=\{j_1,\ldots,j_\ell\}$ and 
  $\mathcal{K}=\supp(\beta)\backslash\mathcal{S}=\{k_1,\ldots,k_\ell\}$.  Let
$\alpha^0=\alpha$ and for $i=1,\ldots,2\ell$ let $\alpha^i$ be 
  the vertex in $V(\Gamma)$ with \[\supp(\alpha^i)=\left\{\begin{array}{ll}
  \supp(\alpha^{i-1})\backslash\{j_{(i+1)/2}\}&\textnormal{if $i$ is odd}\\ 
  \supp(\alpha^{i-1})\cup\{k_{i/2}\}&\textnormal{if $i$ is
even.}\end{array}\right.\]
  It follows that $\wt(\alpha^i)=(m-1)/2$ or $(m+1)/2$ if $i$ is odd or even
respectively.  Moreover, 
  $d(\alpha^i,\alpha^{i-1})=1$ for $i=1,\ldots,2\ell$.  Thus
$$\alpha=\alpha^0,\alpha^1,\ldots,\alpha^{2\ell}=\beta$$
  is a path in $\Delta_2$ from $\alpha$ to $\beta$.  A similar argument shows
that there exists a path in $\Delta_2$
  between two vertices of weight $(m-1)/2$.  Now suppose
$\alpha,\beta\in\Delta_2$ are such that they have different weights with, say, 
  $\alpha$ having weight $(m-1)/2$.  Let
$k\in\supp(\beta)\backslash\supp(\alpha)$ and $\alpha^1$ be such that 
  $\supp(\alpha^1)=\supp(\alpha)\cup\{k\}$.  Then $\alpha^1$ is adjacent to
$\alpha$ and has weight $(m+1)/2$, 
  and as we have just shown, there exists a path in $\Delta_2$ from $\alpha^1$
to $\beta$.     
\end{proof}

\begin{theorem}\label{main2} Let $C$ be a diagonally $X$-neighbour transitive
code in
  $\Gamma=H(m,q)$.  Then one of the following holds:
\begin{itemize}
\item[(i)] $C=\{(a,\ldots,a)\}$ for some $a\in Q$;  
\item[(ii)] $C=\Rep(m,q)$;
\item[(iii)] $C=\Inj(m,q)$ where $m<q$;
\item[(iv)] $C=W([m/2],2)$ where $m\geq 3$ and odd;
\item[(v)] there exists a positive integer $p$ such that $m=pq$ and
  $C$ is contained in $\All(pq,q)$. 
\end{itemize}
\end{theorem}

\begin{proof}  Let $\alpha\in C$ and suppose that $\alpha$ has
  composition $$Q(\alpha)=\{(a_1,p_1),\ldots,(a_k,p_k)\}$$ with
  $p_1\geq p_2\geq\ldots\geq p_k$ and $k\leq q$.  Let
  $H=\Diag_m(S_q)\rtimes L$.  We break our analysis up into the
  cases $k=1$ and $k\geq 2$.   

 {\underline{\bf{Case $k=1$:}}}  In this case
 $\alpha=(a_1,\ldots,a_1)$ and $$C=\alpha^X\subseteq
 \alpha^{H}=\Rep(m,q).$$ If $|C|=1$, then  
 $X\leq H_\alpha=\Diag_m(S_{q-1})\rtimes L$ and 
 $C_1=\{\nu(\alpha,i,b)\,:\,1\leq i\leq m,\,b\in
 Q\backslash\{a_1\}\}$.  As $H_\alpha$ fixes setwise $C$ and $C_1$,
 and is transitive on both, it follows that $C$ is
 $H_\alpha$-neighbour transitive.  By the above reduction we only
 find $C=\{(a_1,\ldots,a_1)\}$, but of course the examples here are
 $\{(a,\ldots,a)\}$ for all $a\in Q$, as in (i).  Suppose now that
 $|C|\geq 2$. Since $C\subseteq\Rep(m,q)$ it follows that $\delta=m$.
 By Remark \ref{regrem}, $C$ is $1$-regular, and because $\delta=m$,
 $C$ is equivalent to $\Rep(m,q)$ by \cite[Sec. 2]{famctr}.  Thus
 $|C|=q$ and $C=\Rep(m,q)$, as in (ii).   

 {\underline{\bf{Case $k\geq 2$:}}}  Suppose first that $p_1=1$.  Then
 $k=m$ and \[\alpha\in\hat{C}=\left\{\begin{array}{ll}
 \All(q,q)&\textnormal{if $m=q$}\\\Inj(m,q)&\textnormal{if
   $m<q$.} \end{array}\right.\]  Since $H$ fixes $\hat{C}$ and
 $X\leq H$, we have that $C=\alpha^X\subseteq\alpha^{H}=\hat{C}.$  If $m=q$
 then (v) holds.  Thus assume that $m<q$ and $\hat{C}=\Inj(m,q)$.  In
 this case $C_1$ contains $\nu=\nu(\alpha,m,\alpha_1)$ and
 $\Num(\nu)=\{(2,1),(1,m-2)\}$.  By Corollary \ref{specprop},
 $\Num(\nu')=\Num(\nu)$ for all $\nu'\in C_1$, and in particular,
 $C_1\cap\hat{C}=\emptyset$.  If $C$ is a proper subset of $\hat{C}$
 then, by Lemmas \ref{connected} and \ref{delta1con}, we have that
 $C_1\cap\hat{C}\neq\emptyset$, which is a contradiction.  Thus
 $C=\Inj(m,q)$ and (iii) holds.     

 We can now assume that $p_1\geq 2$.  As $S_m$ acts $m$-transitively,
 there exists $\sigma\in L$ such that
 $\alpha^\sigma=(a_1^{p_1},\ldots,a_k^{p_k})\in C^\sigma$.  By Lemma
 \ref{distpart}, $C^\sigma$ is $X^{\sigma}$-neighbour
 transitive, and as $\Diag_m(S_q)$ is centralised by $L$, it
 follows that $X^\sigma\leq H$.  Let
 $\bar{X}=X^\sigma$, $\bar{\alpha}=\alpha^\sigma$ and
 $\bar{C}=C^\sigma$.  Suppose that $k<q$.  Then $q\geq 3$ and there 
 exists $a\in Q$ that does not occur in $\bar{\alpha}$.  Consider
 $\nu_1=(a,a_1^{(p_1-1)},a_2^{p_2},\ldots,a_k^{p_k})$ and
 $\nu_2=(a_1^{(p_1+1)},a_2^{(p_2-1)},\ldots,a_k^{p_k})$, which are both
 adjacent to $\bar{\alpha}$.  Then $\Num(\nu_1)$, $\Num(\nu_2)$ and
 $\Num(\bar{\alpha})$ are pairwise distinct, which is a contradiction
 to Corollary \ref{specprop}.  Thus $k=q$.  If $p_j=p_1$ for all $j$,
 then $m=pq$ (where $p=p_1$) and $\Num(\bar{\alpha})=\{(p,q)\}$. 
 Thus $\bar{\alpha}\in\All(pq,q)$
 and $$\bar{C}=\bar{\alpha}^{\bar{X}}\subseteq\bar{\alpha}^{H}=\All(pq,q).$$ 
 As $\sigma\in\Aut(\All(pq,q))$, it follows that
 $C=\bar{C}^{\sigma^{-1}}\subseteq\All(pq,q)$ and (v) holds.  Thus we
 now assume that $p_1> p_k$.  Let $t$ be minimal such that $p_1> p_t$,
 that is, $p=p_1=p_2=\ldots=p_{t-1}>p_t$, and note that $t\geq 2$.
 Define $\nu_1\in\Gamma_1(\bar{\alpha})$
 by \[\nu_1=\left\{\begin{array}{ll}
 (a_1^p,\ldots,a_{t-2}^p,a_{t-1}^{p+1},a_t^{p_t-1},\ldots,a_q^{p_q})&\textnormal
{if
   $t\geq
   3$}\\
(a_1^{p+1},a_t^{p_t-1},a_{t+1}^{p_{t+1}},\ldots,a_q^{p_q})&\textnormal{if
   $t=2$} \end{array}\right.\] and note that $(p+1,1)\in\Num(\nu_1)$
 for all $t$, and $(p,t-2)\in\Num(\nu_1)$ if $t\geq 3$, while no
 element of $\Num(\nu_1)$ has first entry $p$ if $t=2$.  As
 $(p,t-1)\in\Num(\bar{\alpha})$ it follows that
 $\Num(\nu_1)\neq\Num(\bar{\alpha})$, and so Corollary \ref{specprop}
 implies that $\nu_1\in \bar{C}_1$.  We claim that $t=2$,
 $p_t=p_2=p-1$ and $q=2$.      

 Assume to the contrary that the claim is false.  Then $t$, $p_2$, $q$
 satisfy the conditions in column $2$ of Table \ref{tableneigh} for
 exactly one of the lines.  For each line of Table \ref{tableneigh},
 let $\nu_2$ be the vertex in column $3$.  In each case
 $\nu_2\in\Gamma_1(\bar{\alpha})$ and $\Num(\nu_2)\neq\Num(\bar{\alpha})$.  We
also have that
 $\Num(\nu_1)\neq\Num(\nu_2)$: this is clear in lines $2$ and $3$ 
 since then no element of $\Num(\nu_2)$ has first entry $p+1$, while
 in line $1$, $(p,t-3)\in\Num(\nu_2)$ if $t>3$ and no entry of
 $\Num(\nu_2)$ has first entry $p$ if $t=3$.  Since
 $\Num(\nu_2)\neq\Num(\bar{\alpha})$, it follows from Corollary
 \ref{specprop} that $\nu_2\in C_1$.  However, Corollary
 \ref{specprop} then implies that $\Num(\nu_2)=\Num(\nu_1)$, which is
 a contradiction.  Thus the claim is proved. \begin{table}
\caption{Neighbours of $\bar{\alpha}$}
\label{tableneigh}  
\begin{tabular}{lll}
\hline\noalign{\smallskip}
Line & Case & $\nu_2\in\Gamma_1(\bar{\alpha})$  \\
\noalign{\smallskip}\hline\noalign{\smallskip}
1&$t>2$&$(a_1^{p+1},a_2^{p-1},a_3^{p},\ldots,a_{t-1}^p,a_t^{p_t},\ldots,a_q^{p_q
})$\\
2&$t=2$, $p_2\leq p-2$&$(a_1^{p-1},a_2^{p_2+1},a_3^{p_3},\ldots,a_q^{p_q})$\\
3&$t=2$, $p_2=p-1$, $q\geq 3$&$(a_1^p,a_2^p,a_3^{p_3-1},\ldots,a_q^{p_q})$\\
\noalign{\smallskip}\hline
\end{tabular}
\end{table}
As $t=2$, $p_2=p-1$ and $q=2$, it follows that $m=2p-1\geq 3$ and
$\bar{\alpha}=(a_1^p,a_2^{p-1})$.  By identifying $Q$ with $\{0,1\}$,
it follows that $\bar{\alpha}$ has weight $p=(m+1)/2$ or
$p-1=(m-1)/2$, and therefore so does 
$\alpha=\bar{\alpha}^{\sigma^{-1}}$, since $\sigma\in L$.  Thus
$\alpha \in W([m/2],2)$ and $$C=\alpha^X\subseteq
\alpha^{H}=W([m/2],2).$$  Let $\nu\in\Gamma_1(\alpha)$.  Then $\nu$
has weight $(m+3)/2$ or $(m-3)/2$ and
$\Num(\nu)=\{((m+3)/2,1),((m-3)/2,1)\}$.  Thus
$\Num(\nu)\neq\Num(\alpha)$ and Corollary \ref{specprop} implies that
$\nu\in C_1$.  Hence Corollary \ref{specprop} implies that
$\Num(\nu')=\Num(\nu)$ for all $\nu'\in C_1$, in particular $C_1\cap
W([m/2],2)=\emptyset$.  If $C$ is a proper subset of $W([m/2],2)$
then, by Lemmas \ref{connected} and \ref{delta1con}, $C_1\cap
W([m/2],2)\neq\emptyset$, which is a contradiction.  Thus
$C=W([m/2],2)$ and (iv) holds.
\end{proof}

\begin{remark} Theorem \ref{main2} gives us a proof of Theorem
  \ref{main1}. None of the codes in cases (i)--(iv) of Theorem
  \ref{main2} are constant composition codes, and any subset of
  $\All(pq,q)$ is necessarily a frequency permutation array. 
\end{remark}

\section{Neighbour transitive frequency permutation arrays}\label{sec1codect}

We first consider frequency permutation arrays for which each letter from
the alphabet $Q$ appears exactly once in each codeword.  
Such codes are known as \emph{permutation codes}. Permutation codes were first
examined in the mid 1960s and 1970s \cite{blake,blakeetal,frankl,slepian}, but there
has been renewed interest due to the possible applications in powerline
communication, see \cite{bailey,chu1,keevash,smith2} for example.  

In order to describe permutation codes, we identify the alphabet $Q$ with the
set $\{1,\ldots,q\}$ 
and consider codes in the Hamming graph $\Gamma=H(q,q)$.  For $g\in S_q$ 
we define the vertex $$\alpha(g)=(1^g,\ldots,q^g)\in V(\Gamma).$$  Recall that
for a subset
$T\subseteq S_q$, we define the \emph{permutation code generated by $T$} to be
the code $$C(T)=\{\alpha(g)\in V(\Gamma)\,:\,g\in T\}.$$  For a permutation
$g\in
S_q$, the \emph{fixed point set of $g$} is the set $\fix(g)=\{a\in
Q\,:\,a^g=a\}$, and the \emph{degree of $g$} is equal to
$\deg(g)=q-|\fix(g)|$. For $g,h\in S_q$, it is known that
$d(\alpha(g),\alpha(h))=\deg(g^{-1}h)$ 
\cite{bailey}.  Thus, for $T\subseteq S_q$,
it holds that $C(T)$ has minimum distance
$\delta=\min\{\deg(g^{-1}h)\,:\,g,h\in T\,,\,g\neq h\}$, and if $T$ is a group,
this is called 
the \emph{minimal degree of $T$} \cite{blakeetal}.        

Recall that the Hamming graph $\Gamma$ has automorphism group
$\Aut(\Gamma)=B\rtimes L$ where $B\cong S_q^q$ and
$L\cong S_q$.  To distinguish between automorphisms of $\Gamma$ and permutations
in $S_q$, we introduce the
following notation.  For $y\in S_q$ we let
$x_y=(y,\ldots,y)\in B$, and we let $\sigma(y)$ denote the
automorphism induced by $y$ in $L$.  For $\alpha(g)\in V(\Gamma)$,
$$\alpha(g)^{x_y}=(1^g,\ldots,q^g)^{(y,\ldots,y)}=(1^{gy},\ldots,q^{gy}
)=\alpha(gy).$$
Now, suppose that $i^y=j$ for $i,j\in Q$.  Then, by considering
$\alpha(g)$ as the $q$-tuple $(\alpha_1,\ldots,\alpha_q)$, it holds 
that $\alpha(g)^{\sigma(y)}|_j=\alpha_i=i^g=j^{y^{-1}g}$.  Thus
$\alpha(g)^{\sigma(y)}=\alpha(y^{-1}g)$, 
proving Lemma \ref{start}.       

\begin{lemma}\label{start} Let $\alpha(g)\in V(\Gamma)$ and $y\in S_q$.
  Then $\alpha(g)^{x_y}=\alpha(gy)$ and
  $\alpha(g)^{\sigma(y)}=\alpha(y^{-1}g)$.
\end{lemma}

Recall from Remark \ref{regrem} that neighbour transitive codes are $1$-regular.
It turns out that there exists exactly one $1$-regular permutation code with
minimum distance $\delta=2$.  Before we prove this we introduce the following concepts.  We regard $1\in Q$ as
the analogue of zero from linear codes, and define the \emph{weight} of a vertex 
$\beta\in V(\Gamma)$ to be $d(\alpha,\beta)$, where $\alpha=(1,\ldots,1)\in
V(\Gamma)$.  For $\beta=(\beta_i), \gamma=(\gamma_i) \in V(\Gamma)$, 
we say $\beta$ is \emph{covered} by $\gamma$ if $\beta_i=\gamma_i$ for each $i$
such that $\beta_i\neq 1$.  Furthermore, we say that a non-empty set $\mathcal{D}$ of
vertices of weight $k$ in $H(q,q)$ is a \emph{$q$-ary $t-(q,k,\lambda)$ design} 
if for every vertex $\nu$ of weight $t$, there exist exactly $\lambda$ vertices in $\mathcal{D}$ that cover $\nu$.  

\begin{lemma}\label{1reg} Let $T$ be a subset of $S_q$.  Then $C(T)$ is
$1$-regular with $\delta=2$ if and only if $T=S_q$.    
\end{lemma}

\begin{proof}  The reverse direction follows from Theorem \ref{codethm} and observing
that $\All(q,q)=C(S_q)$.  To prove the converse, we first claim that there exists a positive integer $\lambda$
such that $|\Gamma_2(\alpha(t))\cap C(T)|=q(q-1)\lambda/2$ for all $\alpha(t)\in C(T)$.
The code $C(T)$ is equivalent to a $1$-regular code $C$ with minimum distance
$2$ that contains $\alpha=(1,\ldots,1)$.  By interpreting a result of Goethals and van Tilborg \cite[Thm. 9]{upc}, it
follows that $\Gamma_2(\alpha)\cap C$ forms a $q$-ary $1-(q,2,\lambda)$ design for some
positive integer $\lambda$.  By counting the pairs $(\nu,\beta)\in \Gamma_1(\alpha)\times (\Gamma_2(\alpha)\cap
C)$ such that $\beta$ covers $\nu$, we deduce that $|\Gamma_2(\alpha)\cap C|=q(q-1)\lambda/2$.  
As $C$ is $1$-regular, this holds for all codewords $\beta\in C$. 
Furthermore, this property is also preserved by equivalence, so the claim holds.    

Let $\alpha(g_1)\in C(T)$ and $S=\Gamma_2(\alpha(g_1))\cap C(T)$.  As $C(T)$
is $1$-regular with $\delta=2$, it follows that $S \neq\emptyset$.  Let
$\alpha(g_2)\in S$.  Then $d(\alpha(g_2g^{-1}_1),\alpha(1))=2$, and
so $g_2g^{-1}_1=t'$ is a transposition.  Consequently, for each $\alpha(g)\in
S$ there exists a transposition $t\in S_q$ such that $g=tg_1$.  There are
exactly $q(q-1)/2$ transpositions in $S_q$, so $|S|\leq q(q-1)/2$.  However, by the above claim, $|S|\geq q(q-1)/2$.  
Thus $S=\{\alpha(tg_1):\,t\textnormal{ is a transposition in $S_q$}\}.$ Any
permutation can be written as a product of transpositions, so for $g\in T$ we have that $g=t_1t_2\ldots t_\ell$ for some transpositions
$t_1,\ldots,t_{\ell}\in S_q$.  We have just shown that $t_1g=t_1t_1t_2\ldots t_{\ell}=t_2\ldots t_{\ell}\in T$.  Repeating this
argument, we first deduce that $1\in T$, and then that every permutation is in $T$.   
\end{proof}

Let $T$ be a subgroup of $S_q$.  As any group has a regular action on itself by right multiplication, 
it follows from Lemma \ref{start} that $\Diag_q(T)=\{x_y\,:\,y\in T\}$ acts
regularly on $C(T)$.  We also define $$A(T)=\{x_y\sigma(y)\,:\,y\in N_{S_q}(T)\},$$  where 
$N_{S_q}(T)=\{y\in S_q\,:\,T^y=T\}$.  For $x_y\sigma(y)\in A(T)$,
Lemma \ref{start} implies that $\alpha(t)^{x_y\sigma(y)}=\alpha(y^{-1}ty)$ for all $\alpha(t)\in C(T)$.  As
$y\in N_{S_q}(T)$, we deduce that $A(T)\leq\Aut(C(T))_{\alpha(1)}$.  We now prove Theorem \ref{permiff}.

\begin{proof} Suppose that $C(T)$ is diagonally
$X$-neighbour transitive in $H(q,q)$, 
and suppose first that $\delta=2$.  By Remark \ref{regrem}, $C(T)$ is
$1$-regular, and so Lemma \ref{1reg} implies 
that $T=S_q$.  In this case $N_{S_q}(S_q)=S_q$ is $2$-transitive.  Now suppose
that $\delta\geq 3$, and consider the neighbours $\nu(\alpha(1),i_1,i_2)$, $\nu(\alpha(1),j_1,j_2)$
for $i_1\neq i_2$ and $j_1\neq j_2$.
There exists $x=x_y\sigma(z)\in X$ such that
$\nu(\alpha(1),i_1,i_2)^x=\nu(\alpha(1),j_1,j_2)$, 
and as $x\in\Aut(C(T))$, it follows that $\alpha(t)^x\in C(T)$ for all
$\alpha(t)\in T$.
By Lemma \ref{start}, $\alpha(t)^x=\alpha(z^{-1}ty)$, so $z^{-1}ty\in T$ for all
$t\in T$.  In particular, since $T$ is a subgroup, $z^{-1}y\in T$, and so $y^{-1}z\in T$. Hence
$y^{-1}zz^{-1}ty=y^{-1}ty\in T$ for all $t\in T$, that
is, $y\in N_{S_q}(T)$.  Since $y^{-1}z\in T$ it follows that $z\in N_{S_q}(T)$. 
By Lemma \ref{neigact}, 
$\nu(\alpha(1),i_1,i_2)^x=\nu(\alpha(z^{-1}y),i_1^z,i_2^y)$, and because
$\delta\geq 3$ it follows that
$\alpha(z^{-1}y)=\alpha(1)$.  Thus $z=y$, $i_1^z=j_1$ and $i_2^z=j_2$.  In
particular, $N_{S_q}(T)$ acts $2$-transitively on $Q$.  

Now assume that $N_{S_q}(T)$ is $2$-transitive, and let $X=\langle
A(T),\Diag_q(T)\rangle$.  As $\Diag_q(T)$
acts regularly on $C(T)$, it follows that $X$ acts transitively on $C(T)$.  
Consider $\nu(\alpha(1),i_1,i_2)$, $\nu(\alpha(1),j_1,j_2)\in\Gamma_1(\alpha(1))$.  As $N_{S_q}(T)$ is
$2$-transitive, there exists 
$y\in N_{S_q}(T)$ such that $i_1^y=j_1$ and $i_2^y=j_2$.  Let $x=x_y\sigma(y)\in
A(T)$.  By Lemma \ref{neigact},
$\nu(\alpha(1),i_1,i_2)^x=\nu(\alpha(y^{-1}y),i_1^y,i_2^y)=\nu(\alpha(1),j_1,
j_2)$.  Thus $A(T)$ acts transitively
on $\Gamma_1(\alpha(1))$.  Because $X$ acts transitively on $C(T)$, we deduce
that $X$ acts transitively on the
set of neighbours of $C(T)$.  This proves the first statement in Theorem
\ref{permiff}.

Finally suppose that $C(T)$ is a diagonally neighbour transitive code in
$H(q,q)$ and let $p$ be 
a positive integer.  By the previous argument it follows 
that $N_{S_q}(T)$ is $2$-transitive and $C(T)$ is $X$-neighbour transitive with 
$X=\langle A(T),\Diag_q(T)\rangle$.  Moreover $X_{\alpha(1)}=A(T)$ 
acts transitively on $\Gamma_1(\alpha(1))$.  Thus, by Proposition
\ref{replemma}, $\Rep_p(C(T))$ is $(X\times S_p)$-neighbour
transitive in $H(pq,q)$, and because $X\leq \Diag_q(S_q)\rtimes L$ it follows
that $X\times S_p\leq\Diag_{pq}(S_q)\rtimes S_{pq}$.
\end{proof}

\section{Acknowledgements}This research was supported by the Australian
  Research Council Federation Fellowship FF0776186 of the second
  author and also, for the first author, by an Australian Postgraduate Award. 


\end{document}